\documentclass{amsart}

\usepackage[all,cmtip]{xy}

\newtheorem{thm}{Theorem}[section]
\newtheorem{cor}[thm]{Corollary}
\newtheorem{lem}[thm]{Lemma}
\newtheorem{prop}[thm]{Proposition}
\theoremstyle{definition}
\newtheorem{defn}[thm]{Definition}
\theoremstyle{remark}
\newtheorem{rem}[thm]{Remark}
\numberwithin{equation}{section}

\newcommand{\add}{\operatorname{add}}

\newcommand{\End}{\operatorname{End}}
\newcommand{\Ext}{\operatorname{Ext}}
\newcommand{\Hom}{\operatorname{Hom}}
\newcommand{\Id}{\operatorname{Id}}

\newcommand{\proj}{\operatorname{proj}}
\newcommand{\rep}{\operatorname{rep}}
\newcommand{\Tr}{\operatorname{Tr}}

\newcommand{\CC}{\mathcal{C}}
\newcommand{\CE}{\mathcal{E}}
\newcommand{\CI}{\mathcal{I}}
\newcommand{\CP}{\mathcal{P}}
\newcommand{\CQ}{\mathcal{Q}}
\newcommand{\CS}{\mathcal{S}}
\newcommand{\Extc}{\Ext_\mathcal{E}^1}
\newcommand{\Homc}{\Hom_\mathcal{C}}
\newcommand{\Homco}{\overline\Hom_\mathcal{C}}
\newcommand{\Homcu}{\underline\Hom_\mathcal{C}}
\newcommand{\op}{\mathrm{op}}
\newcommand{\set}[1]{\left\{#1\right\}}

\title{The generalized Auslander-Reiten duality on an exact category}
\author{Pengjie Jiao}
\address{School of Mathematical Sciences, University of Science and Technology of China, Hefei 230026, PR China}
\email{jiaopjie@mail.ustc.edu.cn}

\date{\today}

\begin{document}

\begin{abstract}
  We introduce a notion of generalized Auslander-Reiten duality on a Hom-finite Krull-Schmidt exact category $\mathcal{C}$. This duality induces the generalized Auslander-Reiten translation functors $\tau$ and $\tau^-$. They are mutually quasi-inverse equivalences between the stable categories of two full subcategories $\mathcal{C}_r$ and $\mathcal{C}_l$ of $\mathcal{C}$. A non-projective indecomposable object lies in the domain of $\tau$ if and only if it appears as the third term of an almost split conflation; dually, a non-injective indecomposable object lies in the domain of $\tau^-$ if and only if it appears as the first term of an almost split conflation. We study the generalized Auslander-Reiten duality on the category of finitely presented representations of locally finite interval-finite quivers.
\end{abstract}

\maketitle

\section{Introduction}

Throughout, $k$ denotes a commutative artinian ring.
We consider $k$-linear skeletally small exact categories.

Recall that an abelian category $\mathcal{A}$ is said to \emph{have enough almost split sequences} provided that each non-projective indecomposable object appears as the third term of an almost split sequence, and that each non-injective indecomposable object appears as the first term of an almost split sequence.

H.~Lenzing and R.~Zuazua proved that an Ext-finite abelian category $\mathcal{A}$ has enough almost split sequences if and only if it has Auslander-Reiten duality;
see \cite[Theorem~1.1]{LenzingZuazua2004Auslander}.
S.~Liu, P.~Ng and C.~Paquette \cite{LiuNgPaquette2013Almost} investigated almost split sequences in an exact category, and proved the local version of \cite[Theorem~1.1]{LenzingZuazua2004Auslander} under weaker hypotheses;
see \cite[Theorem~3.6]{LiuNgPaquette2013Almost}.
One can observe that the notion of Auslander-Reiten duality applies to exact categories naturally.

Inspired by \cite{Chen2011Generalized}, we introduce the notion of \emph{generalized Auslander-Reiten duality} for a Hom-finite Krull-Schmidt exact category $\CC$. This duality induces the generalized Auslander-Reiten translation functors $\tau$ and $\tau^-$, which are defined on the stable categories of two full subcategories  $\CC_r$ and $\CC_l$ of $\CC$. Then $\CC$ has Auslander-Reiten duality in the sense of \cite{LenzingZuazua2004Auslander} if and only if $\CC_l=\CC=\CC_r$. In general, $\CC_r$ and $\CC_l$ are not equal to $\CC$.

We prove that a non-projective indecomposable object lies in $\CC_r$ if and only if it appears as the third term of an almost split conflation, and that a non-injective indecomposable object lies in $\CC_l$ if and only if it appears as the first term of an almost split conflation; see Proposition~\ref{prop:C_r}. We prove that the generalized Auslander-Reiten translation functors $\tau$ and $\tau^-$ are mutually quasi-inverse equivalences between the projectively stable category of $\CC_r$ and the injectively stable category of $\CC_l$; see Proposition~\ref{prop:quasi-inverse}.

In Section 4, we describe the full subcategories $\CC_r$ and $\CC_l$, and the generalized Auslander-Reiten translation functors $\tau$ and $\tau^-$ in the case that $\CC$ is the category of finitely presented representations of a locally finite interval-finite quiver. This description depends on the results of \cite{BautistaLiuPaquette2013Representation}.

\section{Two full subcategories}

Let $k$ be a commutative artinian ring and $\check{k}$ be the minimal injective cogenerator. Denote by $k$-mod the category of finitely generated $k$-modules and by $D=\Hom_k(-,\check{k})$ the Matlis duality.

Let $\CC$ be a $k$-linear exact category. Recall that an \emph{exact category} is an additive category $\CC$ together with a collection $\CE$ of kernel-cokernel pairs $(i,d)$ which satisfies the axioms in \cite[Appendix~A]{Keller1990Chain}; compare \cite[Section~2]{Quillen1973Higher}. Here, a kernel-cokernel pair $(i,d)$ means a sequence of morphisms $X\xrightarrow{i}Y\xrightarrow{d}Z$ satisfying that $i$ is the kernel of $d$ and $d$ is the cokernel of $i$. A kernel-cokernel pair $(i,d)$ in $\CE$ is called a \emph{conflation}, while $i$ is called an \emph{inflation} and $d$ is called a \emph{deflation}. For two objects $X$ and $Y$ in $\CC$, we denote by $\Extc(X,Y)$ the set of equivalence classes of conflations $Y\to E\to X$. Given a conflation $\mu\colon Y\to E\to X$, for each $f\colon Z\to X$, we denote by $\mu.f=\Extc(f,Y)(\mu)$ the conflation obtained by the pullback of $\mu$ along $f$. Dually, for each $g\colon Y\to Z$, we denote by $g.\mu=\Extc(X,g)(\mu)$ the conflation obtained by the pushout of $\mu$ along $g$.

Recall from \cite[Section~2]{LenzingZuazua2004Auslander} that a morphism $f\colon X\to Y$ is \emph{projectively trivial} provided that for each object $Z$, the induced map $\Extc(f,Z)\colon \Extc(Y,Z)\to \Extc(X,Z)$ is zero. We observe that $f$ is projectively trivial if and only if $f$ factors through each deflation $E\to Y$ ending at $Y$. Dually, a morphism $f$ is \emph{injectively trivial} provided that for each object $Z$, the induced map $\Extc(Z,f)\colon \Extc(Z,X)\to \Extc(Z,Y)$ is zero.

Given two objects $X$ and $Y$, we denote by $\CP(X,Y)$ the set of projectively trivial morphisms from $X$ to $Y$. Then $\CP$ forms an ideal of $\CC$. The \emph{projectively stable category} $\underline\CC$ of $\CC$ is the factor category $\CC/\CP$. Given a morphism $f\colon X\to Y$, we denote by $\underline{f}$ its image in $\underline\CC$. We denote by $\Homcu(X,Y)=\Homc(X,Y)/\CP(X,Y)$ the set of morphisms in $\underline\CC$. Given an object $Z$, we mention that $\Extc(-,Z)$ is a contravariant functor from $\underline{\CC}$ to the category of $k$-modules. We observe that an object $P$ becomes zero in $\underline\CC$ if and only if $P$ is projective in $\CC$.

Dually, we denote by $\CI(X,Y)$ the set of injectively trivial morphisms from $X$ to $Y$. The \emph{injectively stable category} $\overline\CC$ of $\CC$ is the factor category $\CC/\CI$. Given a morphism $f\colon X\to Y$, we denote by $\overline{f}$ its image in $\overline{\CC}$. We denote by $\Homco(X,Y)=\Homc(X,Y)/\CI(X,Y)$ the set of morphisms in $\overline{\CC}$. We mention that $\Extc(Z,-)$ is a functor from $\overline{\CC}$ to the category of $k$-modules. We observe that an object $I$ becomes zero in $\overline\CC$ if and only if $I$ is injective in $\CC$.

Recall that a morphism $v\colon E\to Y$ is \emph{right almost split} if it is not a retraction and each $f\colon Z\to Y$ which is not a retraction factors through $v$. Dually, a morphism $u\colon X\to E$ is \emph{left almost split} if it is not a section and each $f\colon X\to Z$ which is not a section factors through $u$. A conflation $\delta\colon X \xrightarrow{u} E \xrightarrow{v} Y$ is an \emph{almost split conflation} if $u$ is left almost split and $v$ is right almost split. Recall that an object whose endomorphism algebra is local is called \emph{strongly indecomposable}. We mention that given an almost split conflation $X\to E\to Y$, the objects $X$ and $Y$ are strongly indecomposable; see \cite[Proposition~II.4.4]{Auslander1978Functors}. We refer to \cite[Chapter~V]{AuslanderReitenSmalo1995Representation} for the general properties of almost split conflations.

We observe the fact that for each non-split conflation $\mu\colon X\to Y\to Z$, there exists some $\gamma\in D\Extc(Z,X)$ such that $\gamma(\mu)\neq 0$. The following lemma is essentially due to
\cite[Proposition~3.1]{LenzingZuazua2004Auslander};
compare \cite[Theorem~9.3 and Corollary~9.4]{GabrielRoiter1992Representations}
and \cite[Proposition~3.1]{LiuNgPaquette2013Almost}.

\begin{lem}\label{lem:representable}
  Let $\delta\colon X\to E\to Y$ be an almost split conflation and let $\gamma\in D\Extc(Y,X)$ with $\gamma(\delta)\neq 0$. Then we have the following statements.
  \begin{enumerate}
    \item For each $M$, we have a non-degenerate $k$-bilinear map
      \[\langle-,-\rangle_M\colon\Homco(M,X)\times\Extc(Y,M)\longrightarrow\check{k}, \quad(\overline f,\mu)\mapsto\gamma(f.\mu).\]
      If moreover $\Homco(M,X)\in k\textnormal{-mod}$ for each $M$, then the induced map
      \[\phi_{Y,M}\colon\Homco(M,X)\longrightarrow D\Extc(Y,M), \quad\overline{f}\mapsto\langle\overline{f},-\rangle_M,\]
      is an isomorphism and natural in $M$ with $\gamma=\phi_{Y,X}(\overline{\Id_X})$.
    \item For each $M$, we have a non-degenerate $k$-bilinear map
      \[\langle-,-\rangle_M\colon\Extc(M,X)\times\Homcu(Y,M)\longrightarrow\check{k},\quad(\mu,\underline g)\mapsto\gamma(\mu.g).\]
      If moreover $\Homcu(Y,M)\in k\textnormal{-mod}$ for each $M$, then the induced map
      \[\psi_{X,M}\colon\Homcu(Y,M)\longrightarrow D\Extc(M,X), \quad\underline{g}\mapsto\langle-,\underline{g}\rangle_M,\]
      is an isomorphism and natural in $M$ with $\gamma=\psi_{X,Y}(\underline{\Id_Y})$.
  \end{enumerate}
\end{lem}

\begin{proof}
  (1) It is sufficient to show that the $k$-bilinear map $\langle-,-\rangle_M$ is non-degenerate for each $M$. The naturalness of $\phi_Y\colon\Homco(-,X)\to D\Extc(Y,-)$ follows from a direct verification.

  On the one hand, assume that $\mu\colon M\to E'\to Y$ is a non-split conflation. Since $\delta$ is almost split, we obtain the following commutative diagram
  \[\xymatrix{
    \llap{$\mu \colon$ }M\ar[r]\ar@{-->}[d]_-f &E'\ar[r]\ar@{-->}[d] &Y\ar@{=}[d]\\
    \llap{$\delta\colon$ }X\ar[r]          &E \ar[r]       &Y.
  }\]
  Here, the left square is a pushout diagram. Hence $f.\mu=\delta$ and $\langle \overline f,\mu\rangle_M=\gamma(\delta)\ne 0$.

  On the other hand, let $\overline f\colon M\to X$ be a nonzero morphism in $\overline\CC$. We have that $f$ is not injectively trivial in $\CC$. Hence there exists some object $N$ and some $\mu\in\Extc(N,M)$ such that $f.\mu=\Extc(N,f)(\mu)$ is non-split. Since $\delta$ is almost split, it can be obtained by a pullback of $f.\mu$ along some $h\colon Y\to N$.
  \[\xymatrix{
    \llap{$  \mu\colon$ }M\ar[r]\ar[d]_-f   &E_1\ar[r]\ar[d] &N\ar@{=}[d]\\
    \llap{$f.\mu\colon$ }X\ar[r]            &E_2\ar[r]       &N          \\
    \llap{$ \delta\colon$ }X\ar[r]\ar@{=}[u]  &E\ar[r]\ar@{-->}[u]   &Y\ar@{-->}[u]_-h
  }\]
  Hence we have $\delta=(f.\mu).h=f.(\mu.h)$ and then $\langle\overline f,\mu.h\rangle_M=\gamma(\delta)\ne 0$.

  (2) The proof is similar.
\end{proof}

The following lemma is a slight modification of
\cite[Proposition~4.1]{LenzingZuazua2004Auslander};
compare \cite[Theorem~3.6]{LiuNgPaquette2013Almost}.

\begin{lem}\label{lem:ass}
  Let $Y$ be a strongly indecomposable object in $\CC$.
  \begin{enumerate}
    \item Assume $D\Extc(Y,-) \simeq \Homco(-,Y')$ for some $Y'$. If $\Homco(Z,Y') \in k \textnormal{-mod}$ for each $Z$, and $Y'$ has a non-injective strongly indecomposable direct summand, then there exists an almost split conflation ending at $Y$.
    \item Assume $D\Extc(-,Y) \simeq \Homcu(Y',-)$ for some $Y'$. If $\Homcu(Y',Z) \in k \textnormal{-mod}$ for each $Z$, and $Y'$ has a non-projective strongly indecomposable direct summand, then there exists an almost split conflation starting at $Y$.
  \end{enumerate}
\end{lem}

\begin{proof}
  (1) Let $\phi\colon\Homco(-,Y')\to D\Extc(Y,-)$ be an isomorphism of functors. Set $\gamma=\phi_{Y'} (\overline{\Id_{Y'}})$. By the naturalness of $\phi$, for each object $M$ and each morphism $f\colon M\to Y'$, we have
  \[\phi_M(\overline f)=D\Extc(Y,f)(\gamma) =\gamma\circ\Extc(Y,f).\]
  It follows that $\phi_M(\overline f)(\mu)=\gamma(f.\mu)$ for each $\mu\in\Extc(Y,M)$.

  Let $X$ be a non-injective strongly indecomposable direct summand of $Y'$. Then the isomorphism $\phi_X$ induces a non-degenerate $k$-bilinear map
  \[\langle-,-\rangle_X\colon\Homco(X,Y')\times\Extc(Y,X)\longrightarrow\check k, \quad (\overline f,\mu)\mapsto\gamma(f.\mu).\]
  Let $H\subseteq\Homc(X,Y')$ be the subset formed by non-sections. Observe that $H$ is a $k$-submodule since $X$ is strongly indecomposable. We have that $\CI(X,Y')\subseteq H$ since $X$ is non-injective. Hence $H/\CI(X,Y')$ is a proper $k$-submodule of $\Homco(X,Y')$. Then there exists some nonzero $\alpha\in D\Homco(X,Y')$ such that $\alpha(\overline f)=0$ for each $f\in H$. By the non-degenerate bilinear form $\langle-,-\rangle_X$, there exists some non-split conflation $\delta\colon X\xrightarrow{u} E\to Y$ such that
  $\alpha=\langle-,\delta\rangle_X$. Then we have
  $\langle\overline f,\delta\rangle_X=0$ for each $f\in H$.

  We claim that $u$ is left almost split. Indeed, we observe that $u$ is not a section. Assume that $h\colon X\to M$ is not a section. Then for each $g\colon M\to Y'$, the morphism $g\circ h$ is not a section and thus lies in $H$. Hence we have $\langle\overline{g\circ h},\delta\rangle_X = 0$. Consider the non-degenerate $k$-bilinear map
  \[\langle-,-\rangle_M\colon\Homco(M,Y')\times\Extc(Y,M)\longrightarrow\check k, \quad (\overline f,\mu)\mapsto\gamma(f.\mu),\]
  induced by $\phi_M$. For each $g\colon M\to Y'$, we have
  \[\langle\overline g,h.\delta\rangle_M = \gamma(g.(h.\delta)) = \langle\overline{g\circ h}, \delta\rangle_X=0.\]
  This implies that the conflation $h.\delta$ splits. In other words, the morphism $h$ factors through $u$. Therefore the morphism $u$ is left almost split. It follows from \cite[Proposition~II.4.4]{Auslander1978Functors} that $\delta$ is an almost split conflation since $\End_\CC(Y)$ is local.

  (2) The proof is similar.
\end{proof}

From now on, we assume that the $k$-linear exact category is Hom-finite and Krull-Schmidt. Here, the \emph{Hom-finiteness} means that all the Hom $k$-modules are finitely generated. In this case all indecomposable objects are strongly indecomposable. Moreover, the stable categories $\underline\CC$ and $\overline\CC$ are Krull-Schmidt and thus have split idempotents.

We introduce two full subcategories of $\CC$ as follows
\[\CC_r=\set{X\in\CC\middle|\mbox{the functor $D\Extc(X,-)\colon\overline\CC\to k$-mod is representable}}\]
and
\[\CC_l=\set{X\in\CC\middle|\mbox{the functor $D\Extc(-,X)\colon\underline\CC\to k$-mod is representable}}.\]
Here, both $\CC_r$ and $\CC_l$ are additive subcategories which are closed under direct summands.
Indeed, it follows from Yoneda's lemma and the fact that $\underline\CC$ and $\overline\CC$ have split idempotents.

\begin{lem}\label{lem:iso-close}
  Let $X$ and $Y$ be two objects in $\CC$.
  \begin{enumerate}
    \item Assume that $X\in\CC_r$ and $X\simeq Y$ in $\underline\CC$. Then the object $Y$ lies in $\CC_r$.
    \item Assume that $X\in\CC_l$ and $X\simeq Y$ in $\overline\CC$. Then the object $Y$ lies in $\CC_l$.
  \end{enumerate}
\end{lem}

\begin{proof}
  (1) We observe that $D\Extc(X,-)\simeq D\Extc(Y,-)$ as functors. Then the result follows. The proof of (2) is similar.
\end{proof}

As a consequence of Lemmas~\ref{lem:representable} and \ref{lem:ass}, we have the following description of indecomposable objects in $\CC_r$ and $\CC_l$.

\begin{prop}\label{prop:C_r}
  Let $Y$ be an indecomposable object in $\CC$.
  \begin{enumerate}
    \item Assume that $Y$ is non-projective. Then $Y\in\CC_r$ if and only if there exists an almost split conflation ending at $Y$.
    \item Assume that $Y$ is non-injective. Then $Y\in\CC_l$ if and only if there exists an almost split conflation starting at $Y$. \qed
  \end{enumerate}
\end{prop}

\section{An equivalence between stable subcategories}

Let $\CC$ be a Hom-finite Krull-Schmidt exact category. For each object $Y$ in $\CC_r$, we fix some isomorphism of functors
\[\phi_Y\colon\Homco(-,\tau Y)\longrightarrow D\Extc(Y,-).\]
Then $\tau$ gives a map from the objects of $\CC_r$ to $\CC$. Dually, for each object $X$ in $\CC_l$, we fix some isomorphism of functors
\[\psi_X\colon\Homcu(\tau^-X,-)\longrightarrow D\Extc(-,X).\]
Then $\tau^-$ gives a map from the objects of $\CC_l$ to $\CC$.

\begin{lem}\label{lem:tau(iso)}
  Let $X$ and $Y$ be two objects in $\CC$.
  \begin{enumerate}
    \item If $X,Y\in\CC_r$ and $X\simeq Y$ in $\underline\CC$, then we have $\tau X\simeq\tau Y$ in $\overline\CC$.
    \item If $X,Y\in\CC_l$ and $X\simeq Y$ in $\overline\CC$, then we have $\tau^-X\simeq\tau^-Y$ in $\underline\CC$.
  \end{enumerate}
\end{lem}

\begin{proof}
  (1) We observe that $\Homco(-,\tau X)\simeq\Homco(-,\tau Y)$, since they are both isomorphic to $D\Extc(X,-)\simeq D\Extc(Y,-)$. Then the result follows from Yoneda's lemma. The proof of (2) is similar.
\end{proof}

\begin{lem}\label{lem:tau^-tauY=Y}
  Let $Y$ be an object in $\CC$.
  \begin{enumerate}
    \item If $Y\in\CC_r$, then $\tau Y\in\CC_l$ and $Y\simeq\tau^-\tau Y$ in $\underline\CC$.
    \item If $Y\in\CC_l$, then $\tau^-Y\in\CC_r$ and $Y\simeq\tau\tau^-Y$ in $\overline\CC$.
  \end{enumerate}
\end{lem}
\begin{proof}
  We only prove (1). We may assume that $Y$ is indecomposable and non-projective. By Lemma~\ref{lem:ass}(1), there exists some almost split conflation $X\to E\to Y$. Then we have $\Homco(-,X)\simeq D\Extc(Y,-)$ and $\Homcu(Y,-)\simeq D\Extc(-,X)$ by Lemma~\ref{lem:representable}. We then obtain $X\in\CC_l$. It follows from Yoneda's lemma that $\tau Y\simeq X$ in $\overline\CC$, and $\tau^-X\simeq Y$ in $\underline\CC$. Hence we have $\tau Y\in\CC_l$ by Lemma~\ref{lem:iso-close}(2). Then we have $\tau^-\tau Y\simeq\tau^-X\simeq Y$ in $\underline\CC$. Here, the first isomorphism follows from Lemma~\ref{lem:tau(iso)}(2).
\end{proof}

We denote by $\underline{\CC_r}$ the image of $\CC_r$ under the canonical functor $\CC\to\underline\CC$, and by $\overline{\CC_l}$ the image of $\CC_l$ under the canonical functor $\CC\to\overline\CC$. We will make $\tau$ into a functor from $\underline{\CC_r}$ to $\overline{\CC_l}$, and $\tau^-$ into a functor from $\overline{\CC_l}$ to $\underline{\CC_r}$.

For each morphism $f\colon Y\to Y'$ in $\CC_r$, define the morphism $\tau(f)\colon\tau Y\to \tau Y'$ in $\overline{\CC_l}$ such that the following diagram commutes
\[\xymatrix{
  \Homco(-,\tau Y)\ar[r]^-{\phi_Y}\ar[d]_-{\Homco(-,\tau(f))}
    &D\Extc(Y ,-)\ar[d]^-{D\Extc(f,-)}\\
  \Homco(-,\tau Y')\ar[r]^-{\phi_{Y'}}
    &D\Extc(Y',-).
}\]
Here, the existence and uniqueness of $\tau(f)$ are guaranteed by Yoneda's lemma. Then it follows that $\tau$ is a functor from $\CC_r$ to $\overline{\CC_l}$. Moreover, if $f$ is projectively trivial, then $D\Extc(f,-)=0$ and thus $\tau(f)=0$ in $\overline{\CC_l}$. Hence $\tau$ induces a functor from  $\underline{\CC_r}$ to $\overline{\CC_l}$ which we still denote by $\tau$.

Similarly, we have a functor $\tau^-\colon\overline{\CC_l}\to\underline{\CC_r}$. For each $\overline g\colon X\to X'$ in $\overline{\CC_l}$, the morphism $\tau^-(\overline g)\colon \tau^-X\to\tau^-X'$ is given by the following commutative diagram
\[\xymatrix{
  \Homcu(\tau^-X',-)\ar[r]^-{\psi_{X'}}\ar[d]_-{\Homcu(\tau^-(\overline g),-)}
    &D\Extc(-,X')\ar[d]^-{D\Extc(-,g)}\\
  \Homcu(\tau^-X,-)\ar[r]^-{\psi_X}
    &D\Extc(-,X).
}\]

We will show that the functors $\tau$ and $\tau^-$ are mutually quasi-inverse equivalences between $\underline{\CC_r}$ and $\overline{\CC_l}$.

For each object $Y\in\underline{\CC_r}$, we set
\[
  \underline{\epsilon_Y} =
  (\psi_{\tau Y,Y}^{-1} \circ \phi_{Y,\tau Y})
  (\overline{\Id_{\tau Y}})
  \colon \tau^-\tau Y \longrightarrow Y
\]
in $\underline{\CC_r}$.
For each object $X\in\overline{\CC_l}$, we set
\[
  \overline{\eta_X} =
  (\phi_{\tau^-X,X}^{-1} \circ \psi_{X,\tau^-X})
  (\underline{\Id_{\tau^-X}})
  \colon X\longrightarrow \tau\tau^-X
\]
in $\overline{\CC_l}$.

\begin{lem}\label{lem:naturalness}
  Let $\underline{\epsilon_Y}$ and $\overline{\eta_X}$ be as above.
  \begin{enumerate}
    \item The morphisms $\underline{\epsilon_Y}$ yield a natural transformation $\underline{\epsilon} \colon \tau^- \tau \to \Id_{\underline{\CC_r}}$, and for each morphism $\underline f\colon Y\to Y'$ in $\underline{\CC_r}$ we have
        \[
          \tau(\underline f) =
          (\phi_{Y',\tau Y}^{-1} \circ \psi_{\tau Y,Y'})
          (\underline f \circ \underline{\epsilon_Y}).
        \]
    \item The morphisms $\overline{\eta_X}$ yield a natural transformation $\overline{\eta} \colon \Id_{\overline{\CC_l}} \to \tau \tau^-$, and for each morphism $\overline g\colon X\to X'$ in $\overline{\CC_l}$ we have
        \[
          \tau^-(\overline g) =
          (\psi_{X,\tau^-X'}^{-1} \circ \phi_{\tau^-X',X})
          (\overline{\eta_{X'}} \circ \overline g).
        \]
  \end{enumerate}
\end{lem}

\begin{proof}
  (1) For each $\underline f\colon Y\to Y'$ in $\underline{\CC_r}$, we have the following commutative diagram
  \[\xymatrix@+1.5em{
    \Homco(\tau Y,\tau Y)\ar[r]^-{\phi_{Y,\tau Y}}\ar[d]_-{\Homco(\tau Y,\tau(\underline f))}
      &D\Extc(Y,\tau Y)\ar[d]_-{D\Extc(f,\tau Y)}
      &\Homcu(\tau^-\tau Y,Y)\ar[l]_-{\psi_{\tau Y,Y}}\ar[d]_-{\Homcu(\tau^-\tau Y,\underline f)}\\
    \Homco(\tau Y,\tau Y')\ar[r]^-{\phi_{Y',\tau Y}}
      &D\Extc(Y',\tau Y)
      &\Homcu(\tau^-\tau Y,Y')\ar[l]_-{\psi_{\tau Y,Y'}}.
  }\]
  The left square commutes by the definition of $\tau(\underline f)$, and the right square commutes since the isomorphism $\psi_{\tau Y}$ is natural. By a diagram chasing, we obtain
  \[
    \tau(\underline f) =
    (\phi_{Y',\tau Y}^{-1} \circ \psi_{\tau Y,Y'})
    (\underline f\circ\underline{\epsilon_Y}).
  \]

  We have the following commutative diagram
  \[\xymatrix@+1.5em{
    \Homco(\tau Y',\tau Y')\ar[r]^-{\phi_{Y',\tau Y'}}\ar[d]_-{\Homco(\tau(\underline f),\tau Y')}
      &D\Extc(Y',\tau Y')\ar[d]_-{D\Extc(Y',\tau(\underline f))}
      &\Homcu(\tau^-\tau Y',Y')\ar[l]_-{\psi_{\tau Y',Y'}}\ar[d]_-{\Homcu(\tau^-\tau(\underline f),Y')}\\
    \Homco(\tau Y,\tau Y')\ar[r]^-{\phi_{Y',\tau Y}}
      &D\Extc(Y',\tau Y)
      &\Homcu(\tau^-\tau Y,Y')\ar[l]_-{\psi_{\tau Y,Y'}}.
  }\]
  The right square commutes by the definition of $\tau^-\tau(\underline f)$. By a diagram chasing, we obtain
  \[
    \tau(\underline f) =
    (\phi_{Y',\tau Y}^{-1} \circ \psi_{\tau Y,Y'})
    (\underline{\epsilon_{Y'}}\circ\tau^-\tau(\underline f)).
  \]
  We then obtain $\underline f\circ\underline{\epsilon_Y}=\underline{\epsilon_{Y'}}\circ\tau^-\tau(\underline f)$. It follows that $\underline\epsilon$ is a natural transformation.

  (2) The proof is similar.
\end{proof}

The following result strengthens \cite[Proposition~3.3]{LenzingZuazua2004Auslander}.

\begin{prop}\label{prop:quasi-inverse}
  The natural transformations $\underline\epsilon$ and $\overline\eta$ are both isomorphisms. Hence, the functors $\tau$ and $\tau^-$ are quasi-inverse to each other.
\end{prop}

\begin{proof}
  We only prove that $\underline{\epsilon_Y}$ is an isomorphism for each $Y\in\CC_r$. We may assume that $Y$ is indecomposable and non-projective in $\CC$. Let $\alpha=\psi_{\tau Y,\tau^-\tau Y}(\underline{\Id_{\tau^-\tau Y}})$ in $D\Extc(\tau^-\tau Y,\tau Y)$ and let $\beta=\phi_{Y,\tau Y}(\overline{\Id_{\tau Y}})$ in $D\Extc(Y,\tau Y)$. By the definition of $\underline{\epsilon_Y}$, we have $\beta = \psi_{\tau Y,Y} (\underline{\epsilon_Y})$.

  Consider the following commutative diagram
  \[\xymatrix@C+1em{
    \Homcu(\tau^-\tau Y,\tau^-\tau Y)\ar[r]^-{\psi_{\tau Y,\tau^-\tau Y}}\ar[d]_-{\Homcu(\tau^-\tau Y,\underline{\epsilon_Y})}
      &D\Extc(\tau^-\tau Y,\tau Y)\ar[d]^-{D\Extc(\epsilon_Y,\tau Y)}\\
    \Homcu(\tau^-\tau Y,Y)\ar[r]^-{\psi_{\tau Y,Y}}
      &D\Extc(Y,\tau Y).
  }\]
  By a diagram chasing, we obtain
  \[\begin{split}
    \beta &= \psi_{\tau Y,Y} (\underline{\epsilon_Y})\\
    &= (\psi_{\tau Y,Y} \circ
      \Homcu(\tau^-\tau Y, \underline{\epsilon_Y}))
      (\underline{\Id_{\tau^-\tau Y}})\\
    &= (D\Extc(\epsilon_Y,\tau Y) \circ \psi_{\tau Y,\tau^-\tau Y})
      (\underline{\Id_{\tau^-\tau Y}})\\
    &= D\Extc(\epsilon_Y,\tau Y) (\alpha)\\
    &= \alpha \circ \Extc(\epsilon_Y, \tau Y).
  \end{split}\]
  Here, the third equality holds by the commutative diagram. The fourth equality holds by the definition of $\alpha$.

  By Lemma~\ref{lem:ass}(1), there exists an almost split conflation $\delta\colon X\to E\to Y$. By Lemma~\ref{lem:representable}(1), we have a natural isomorphism $\phi'\colon\Homco(-,X)\to D\Extc(Y,-)$ such that $\phi'_X(\overline{\Id_X})(\delta)\ne 0$. Setting $\beta'=\phi'_X(\overline{\Id_X})$, we have $\beta'(\delta)\neq 0$. By Yoneda's lemma, there exists some $s\colon X\to\tau Y$ such that $\Homco(-,\overline{s}) = \phi_Y^{-1} \circ \phi'$. We obtain
  \[\beta' = \phi'_X (\overline{\Id_X}) = (\phi_{Y,X} \circ \Homco(X, \overline{s})) (\overline{\Id_X}) = \phi_{Y,X} (\overline{s}).\]

  Consider the following commutative diagram
  \[\xymatrix{
    \Homco(\tau Y,\tau Y)\ar[r]^-{\phi_{Y,\tau Y}} \ar[d]_-{\Homco(\overline{s}, \tau Y)}
      &D\Extc(Y ,\tau Y)\ar[d]^-{D\Extc(Y,s)}\\
    \Homco(X,\tau Y)\ar[r]^-{\phi_{Y,X}}
      &D\Extc(Y,X).
  }\]
  By a diagram chasing, we obtain
  \[\begin{split}
    \beta' &= \phi_{Y,X} (\overline{s})\\
    &= (\phi_{Y,X} \circ \Homco(\overline{s}, \tau Y))
      (\overline{\Id_{\tau Y}})\\
    &= (D\Extc(Y, s) \circ \phi_{Y, \tau Y})
      (\overline{\Id_{\tau Y}})\\
    &= D\Extc(Y, s) (\beta)\\
    &= \beta \circ \Extc(Y, s).
  \end{split}\]
  Here, the third equality holds by the commutative diagram. The fourth equality holds by the definition of $\beta$.
  Since $\beta = \alpha \circ \Extc(\epsilon_Y,\tau Y)$, we obtain
  \[
    \beta' = \alpha \circ \Extc(\epsilon_Y,\tau Y) \circ \Extc(Y,s).
  \]

  Then we have that
  \[
    0 \neq \beta'(\delta)
    = \alpha((s.\delta).\epsilon_Y)
    = \alpha(s.(\delta.\epsilon_Y)).
  \]
  In particular, the conflation $\delta.\epsilon_Y$ is non-split. It follows that $\epsilon_Y\colon\tau^-\tau Y\to Y$ is a retraction in $\CC$, since $\delta$ is almost split. By Lemma~\ref{lem:tau^-tauY=Y}(1), we have that $\tau^-\tau Y\simeq Y$ in $\underline\CC$. It follows that $\underline{\epsilon_Y}$ is an isomorphism in $\underline\CC$.
\end{proof}

The following lemma shows that $(\tau^-,\tau)$ forms an adjoint pair, with unit $\overline\eta$ and counit $\underline\epsilon$; see \cite[Section~IV.1]{MacLane1971Categories}.

\begin{lem}
  We have $\tau(\underline{\epsilon_Y})\circ\overline{\eta_{\tau Y}}=\overline{\Id_{\tau Y}}$ for each $Y\in\underline{\CC_r}$, and $\underline{\epsilon_{\tau^-X}}\circ\tau^-(\overline{\eta_X})=\underline{\Id_{\tau^-X}}$ for each $X\in\overline{\CC_l}$.
\end{lem}

\begin{proof}
  We only prove the first equality. We have the following commutative diagram
  \[\xymatrix@+1em{
    \Homco(\tau Y,\tau\tau^-\tau Y)\ar[r]^-{\phi_{\tau^-\tau Y,\tau Y}}\ar[d]_-{\Homco(\tau Y,\tau(\underline{\epsilon_Y}))}
      &D\Extc(\tau^-\tau Y,\tau Y)\ar[d]_-{D\Extc(\epsilon_Y,\tau Y)}
      &\Homcu(\tau^-\tau Y,\tau^-\tau Y)\ar[l]_-{\psi_{\tau Y,\tau^-\tau Y}}\ar[d]_-{\Homcu(\tau^-\tau Y, \underline{\epsilon_Y})}\\
    \Homco(\tau Y,\tau Y)\ar[r]^-{\phi_{Y,\tau Y}}
      &D\Extc(Y,\tau Y)
      &\Homcu(\tau^-\tau Y,Y)\ar[l]_-{\psi_{\tau Y,Y}}.
  }\]
  The left square commutes by the definition of $\tau(\underline{\epsilon_Y})$, and the right square commutes since $\psi_{\tau Y}$ is natural.
  By a diagram chasing, we obtain
  \[\begin{split}
    \overline{\Id_{\tau Y}}
    &= (\phi_{Y,\tau Y}^{-1} \circ \psi_{\tau Y,Y})
      (\underline{\epsilon_Y})\\
    &= (\phi_{Y,\tau Y}^{-1} \circ \psi_{\tau Y,Y} \circ
      \Homcu(\tau^-\tau Y, \underline{\epsilon_Y}))
      (\underline{\Id_{\tau^-\tau Y}})\\
    &= (\Homco(\tau Y,\tau(\underline{\epsilon_Y})) \circ
      \phi_{Y,\tau Y}^{-1} \circ \psi_{\tau Y,Y})
      (\underline{\Id_{\tau^-\tau Y}})\\
    &= \Homco(\tau Y,\tau(\underline{\epsilon_Y}))
      (\overline{\eta_{\tau Y}})\\
    &= \tau(\underline{\epsilon_Y}) \circ \overline{\eta_{\tau Y}}.
  \end{split}\]
  Here, the first equality holds by the definition of $\underline{\epsilon_Y}$. The third equality holds by the commutative diagram. The fourth equality holds by the definition of $\overline{\eta_{\tau Y}}$.
\end{proof}

\begin{defn}
  Let $\CC$ be a Hom-finite Krull-Schmidt exact category. We call the sextuple obtained above
\[\set{\CC_r, \CC_l, \phi, \psi, \tau, \tau^-}\]
the \emph{generalized Auslander-Reiten duality} on $\CC$ and call the functors $\tau$ and $\tau^-$ the \emph{generalized Auslander-Reiten translation functors}.
\end{defn}

We mention that the functor $\tau$ depends on $\phi$ and the functor $\tau^-$ depends on $\psi$. Then $\CC$ \emph{has Auslander-Reiten duality} in the sense of \cite{LenzingZuazua2004Auslander} if and only if $\CC_l=\CC=\CC_r$.

\begin{rem}
  Let $(\CC,\CE)$ be a Frobenius category. Then the projectively stable category and the injectively stable category of $\CC$ are the same and have a natural triangulated structure. The generalized Auslander-Reiten duality on $\CC$ gives the generalized Serre duality on $\underline\CC$ in the sense of \cite{Chen2011Generalized}. More precisely, we have $(\underline\CC)_r=\underline{\CC_r}$ and $(\underline\CC)_l=\overline{\CC_l}$. Let $\Sigma$ be the translation functor of $\underline\CC$. Then the functor $\Sigma \tau \colon \underline{\CC_r} \to \overline{\CC_l}$ gives the generalized Serre functor of $\underline\CC$.
\end{rem}

\section{The category of finitely presented representations}

From now on, we let $k$ be a field and $Q=(Q_0,Q_1)$ be a quiver. Here, $Q_0$ is the set of vertices and $Q_1$ is the set of arrows. Given an arrow $\alpha\colon a\to b$, denote by $s(\alpha)=a$ its source and by $t(\alpha)=b$ its target.
A path $p$ of length $l\geq1$ is a sequence of arrows $\alpha_l\cdots\alpha_2\alpha_1$ such that $s(\alpha_{i+1})=t(\alpha_i)$ for each $i=1,2,\dots,l-1$. We let $s(p)=s(\alpha_1)$ and $t(p)=t(\alpha_l)$. A \emph{left infinite path} is an infinite sequence of arrows $\alpha_1\alpha_2\cdots\alpha_n\cdots$ such that $s(\alpha_i)=t(\alpha_{i+1})$ for each $i\geq 1$. Dually, a \emph{right infinite path} is an infinite sequence of arrows $\cdots\alpha_n\cdots\alpha_2\alpha_1$ such that $s(\alpha_{i+1})=t(\alpha_i)$ for each $i\geq 1$.

Recall that a quiver $Q$ is \emph{locally finite} if for each $a\in Q_0$, the set of arrows starting at $a$ or ending at $a$ is finite. A quiver $Q$ is \emph{interval-finite} if for any $a,b\in Q_0$, the set of paths $p$ with $s(p)=a$ and $t(p)=b$ is finite. We will assume that $Q$ is locally finite and interval-finite.

A representation $M$ of $Q$ is called \emph{locally finite dimensional} if $M(a)$ is finite dimensional for each $a\in Q_0$, and is \emph{finite dimensional} if moreover $\bigoplus_{a\in Q_0}M(a)$ is finite dimensional. Denote by $\rep(Q)$ the category of locally finite dimensional representations, and by $\rep^b(Q)$ the full subcategory formed by finite dimensional representations. Then the Matlis duality induces a duality $D\colon\rep(Q)\to\rep(Q^{\op})$, which sends $\rep^b(Q)$ into $\rep^b(Q^{\op})$. Here, $Q^{\op}$ is the opposite quiver of $Q$.

Recall from \cite[Section~2]{BongartzGabriel1982Covering} that the \emph{path-category} $\CQ$ of $Q$ has $Q_0$ as the set of objects; if $x,y\in Q_0$, the morphisms from $x$ to $y$ are the linear combinations of paths from $x$ to $y$. It is well known that $\rep(Q)$ is equivalent to the category of covariant functors from $\CQ$ to $k$-mod.
We denote by $\proj(Q)$ the full subcategory of $\rep(Q)$ formed by objects which corresponds to the finite direct sums of representable functors $\Hom_\CQ(x, -)$ for some objects $x \in \CQ$.

A representation $M$ is called \emph{finitely presented}, if there exists an exact sequence $P_1 \xrightarrow{u} P_0 \xrightarrow{v} M \to 0$ with $P_1, P_0 \in \proj(Q)$. The exact sequence is called a \emph{presentation} of $M$; it is a \emph{minimal presentation} if moreover the kernels of $u$ and $v$ are superfluous in $P_1$ and $P_0$, respectively.
Denote by $\rep^+(Q)$ the full subcategory of $\rep(Q)$ formed by the finitely presented representations. We mention that $\rep^+(Q)$ is a Hom-finite hereditary abelian subcategory, which is closed under extensions in $\rep(Q)$; moreover, we have $\rep^b(Q)\subseteq\rep^+(Q)$;
see \cite[Proposition~1.15]{BautistaLiuPaquette2013Representation}.

Denote by $A = kQ$ the path algebra of $Q$. Then $A$ admits a complete set of primitive orthogonal idempotents $\set{e_a \middle| a \in Q_0}$. Recall that a left $A$-module $M$ is called \emph{unitary} if $AM=M$, which is equivalent to $M = \bigoplus_{a\in Q_0} e_a M$.
Then $_A A$ is a left unitary $A$-modules.
It is well known that the category of representations of $Q$ is equivalent to the category of unitary left $A$-modules.
For convenience, we will identify representations of $Q$ with the corresponding unitary left $A$-modules.
We mention that the contravariant functor $\Hom_A(-, A) \colon \proj(Q) \to \proj(Q^\op)$ is a duality;
see \cite[Lemma~1.18]{BautistaLiuPaquette2013Representation}.

Let $P_1 \xrightarrow{g} P_0 \to M \to 0$ be a minimal presentation of a representation $M$ in $\rep(Q)$. The cokernel of $\Hom_A(g,A)$ is called the \emph{transpose} $\Tr M$ of $M$. Here, $\Tr M$ has no nonzero projective direct summands. We observe that a morphism $f\colon X\to Y$ in $\rep^+(Q)$ is projectively trivial if and only if it factors through a projective object, since $\rep^+(Q)$ has enough projectives. Then we obtain a duality $\Tr\colon\underline\rep^+(Q)\to\underline\rep^+(Q^{\op})$. Here, $\underline\rep^+(Q)$ can be embedded in $\rep^+(Q)$ as a full subcategory, since $\rep^+(Q)$ is hereditary. Then we have a contravariant functor
\[\Tr\colon\rep^+(Q)\longrightarrow\rep^+(Q^{\op}).\]

The following lemma is contained in the proof of \cite[Theorem~2.8]{BautistaLiuPaquette2013Representation}.

\begin{lem}\label{lem:AR-formula}
  Let $L,M\in\rep(Q)$.
  \begin{enumerate}
    \item If $M$ lies in $\rep^+(Q)$, then there exists an isomorphism $\Hom(L,D\Tr M)\simeq D\Ext^1(M,L)$, which is natural in $L$ and $M$.
    \item If $M$ lies in $\rep^b(Q)$, then there exists an isomorphism $\Hom(\Tr DM,L)\simeq D\Ext^1(L,M)$, which is natural in $L$ and $M$. \qed
  \end{enumerate}
\end{lem}

The following lemma is due to
\cite[Theorem~2.8, Corollary~2.9 and Propositions~3.6]{BautistaLiuPaquette2013Representation}.

\begin{lem}\label{lem:rep(Q)-ass}
  Let $M\in\rep^+(Q)$ be an indecomposable representation.
  \begin{enumerate}
    \item If $M$ is non-projective, then there exists an almost split sequence $0\to D\Tr M\to E\to M\to 0$ in $\rep(Q)$.
    \item If $M$ is non-injective and lies in $\rep^b(Q)$, then there exists an almost split sequence $0\to M\to E\to \Tr DM\to 0$ in $\rep^+(Q)$.
    \item Assume that $0\to M\to E\to N\to 0$ is an exact sequence in $\rep(Q)$. Then it is an almost split sequence in $\rep^+(Q)$ if and only if it is an almost split sequence in $\rep(Q)$ with $M\in\rep^b(Q)$. \qed
  \end{enumerate}
\end{lem}

\begin{lem}\label{lem:I(X,Y)}
  Let $f\colon X\to Y$ be an injectively trivial morphism in $\rep^+(Q)$. If $Y$ has no nonzero injective direct summands and lies in $\rep^b(Q)$, then we have $f=0$.
\end{lem}

\begin{proof}
  We may assume that $Y$ is indecomposable. By Lemma~\ref{lem:rep(Q)-ass}(2), we have an almost split sequence $0\to Y\to E\to \Tr DY\to 0$ in $\rep^+(Q)$. By Lemma~\ref{lem:AR-formula}(1), we have $\Hom(X,D\Tr\Tr DY)\simeq D\Ext^1(\Tr DY,X)$. By Lemmas~\ref{lem:representable}(1), we have $\overline\Hom(X,Y)\simeq D\Ext^1(\Tr DY,X)$. Observe that $D\Tr\Tr DY\simeq Y$. Then we have $\overline\Hom(X,Y)\simeq\Hom(X,Y)$. Then the result follows.
\end{proof}

Now, we can describe the generalized Auslander-Reiten duality on $\rep^+(Q)$.
Given a collection $\CS$ of objects, denote by $\add\CS$ the category of direct summands of finite direct sums of objects in $\CS$.

\begin{prop}\label{prop:rep(Q)-AR}
  Let $Q$ be a locally finite interval-finite quiver. Set $\CC=\rep^+(Q)$. Then we have
  \[\CC_r=\set{X\in\CC\middle|\Tr X\in\rep^b(Q^{\op})}\]
  and
  \[\CC_l=\add\set{X\middle|\mbox{$X\in\rep^b(Q)$ or $X$ is an injective object in $\CC$}}.\]
  Moreover, the functors $D\Tr$ and $\Tr D$ induce the generalized Auslander-Reiten translation functors.
\end{prop}

Denote by $(\CC_r)_\CP$ the full subcategory of $\CC_r$ formed by objects without nonzero projective direct summands, and by $(\CC_l)_\CI$ the full subcategory of $\CC_l$ formed by objects without nonzero injective direct summands.
We obtain the induced functors $D\Tr\colon(\CC_r)_\CP\to(\CC_l)_\CI$ and $\Tr D\colon(\CC_l)_\CI\to(\CC_r)_\CP$. Since $\CC$ is hereditary and has enough projectives, the canonical functor $(\CC_r)_\CP\to\underline{\CC_r}$ is an equivalence. By Lemma~\ref{lem:I(X,Y)}, the canonical functor $(\CC_l)_\CI\to\overline{\CC_l}$ is an equivalence. Then we have the induced functors $D\Tr\colon\underline{\CC_r}\to\overline{\CC_l}$ and $\Tr D\colon\overline{\CC_l}\to\underline{\CC_r}$.

\begin{proof}
  We observe that $\Tr P=0$ for each projective object $P$. Then the first equality follows from Proposition~\ref{prop:C_r}(1) and Lemma~\ref{lem:rep(Q)-ass}(1) and (3). The second equality follows from Proposition~\ref{prop:C_r}(2) and Lemma~\ref{lem:rep(Q)-ass}(2) and (3).

  Let $X\in\CC$. By Lemmas~\ref{lem:AR-formula}(1) and \ref{lem:I(X,Y)}, for each $Y\in(\CC_r)_\CP$, we have an isomorphism $\overline\Hom(X,D\Tr Y)\simeq D\Ext^1(Y,X)$, which is natural in $X$ and $Y$ since $D\Tr Y$ has no nonzero injective direct summands. Similarly, for each $Y\in(\CC_l)_\CI$, we have $\underline\Hom(\Tr DY,X)\simeq D\Ext^1(X,Y)$, which is natural in $X$ and $Y$. Then the result follows from the construction of generalized Auslander-Reiten translation functors.
\end{proof}

Combining Proposition~\ref{prop:rep(Q)-AR} and \cite[Theorem~3.7]{BautistaLiuPaquette2013Representation}, we have the following direct consequence.

\begin{cor}
  Let $Q$ be a connected locally finite interval-finite quiver. Then the category $\rep^+(Q)$ has Auslander-Reiten duality if and only if $Q$ has neither left infinite path nor right infinite path, or else $Q$ is of the form
  \[\cdots\longrightarrow \circ\longrightarrow \cdots\longrightarrow \circ\longrightarrow \circ. \eqno\qed\]
\end{cor}

\section*{Acknowledgements}

The author thanks his advisor Professor~Xiao-Wu Chen for his guidance and encouragement, and thanks Dawei Shen for some suggestions and comments.


\end{document}